\documentclass{llncs}
\usepackage[utf8]{inputenc}
\usepackage[T1]{fontenc} 
\usepackage[english]{babel}
\usepackage{hyperref, url, amsmath,  amssymb,mathpazo}
\usepackage{amsfonts,amscd,shadow}
\usepackage[ruled]{algorithm2e}
\newcommand{\cmark}{\ding{51}}
\newcommand{\xmark}{\ding{55}}
\usepackage{booktabs}

\usepackage{enumerate}
\usepackage{cleveref}
\usepackage{xcolor}
\usepackage{algpseudocode}
\usepackage{listings}
\usepackage{csquotes}
\usepackage{url}
\usepackage{subcaption}
\usepackage{multirow}
\usepackage{float}
\usepackage{tikz}

\definecolor{light-gray}{gray}{0.95}

\usepackage{tabto}
\usepackage{comment}
\usepackage{pifont} 
\usepackage{stmaryrd}                
\usepackage{graphicx}                  
\usepackage{fancybox}                
\usepackage{fancyhdr}                 
\usepackage{color}                        
\usepackage{variations}                
\usepackage{tikz,tkz-tab}              
\usepackage{colortbl}                    
\usepackage{diagbox}
\usepackage{enumitem}
\usepackage{setspace,letterspace}
\usepackage{pstricks,pstricks-add,pst-math,pst-xkey,pst-text}
\usepackage{csquotes}
\newcommand{\R}{\mathbb{R}}

\newcommand{\N}{\mathbb{N}}

\renewcommand{\nl}{\newline}

\newtheorem{theo}{Theroem}

\newtheorem{lemm}{Lemma}

\title{Continuous Convexity Measures}
\author{Abel Douzal  \and Ferdinand Jacobé de Naurois }
\institute{
    DI\'ENS, \'ENS, PSL University, Paris, France\\
45 rue d'Ulm, 75230, Paris \textsc{cedex} 05, France\\
\email{\url{abel.douzal@ens.fr}} 
\and
    DMA\'ENS, \'ENS, PSL University, Paris, France\\
45 rue d'Ulm, 75230, Paris \textsc{cedex} 05, France\\
\email{\url{ferdinand.jacobe.de.naurois@ens.fr}} 
}

\begin{document}
\maketitle 
\small

\begin{abstract}
Methods for measuring convexity defects of compacts in $\mathbb{R}^n$ abound. However, none of the those measures seems to take into account \textsl{continuity}. Continuity in convexity measure is essential for optimization, stability analysis, global optimality, convergence analysis, and accurate modeling as it ensures robustness and facilitates the development of efficient algorithms for solving convex optimization problems. This paper revisits the axioms underlying convexity measures by enriching them with a continuity hypothesis in Hausdorff's sense. Having provided the concept's theoretical grounds we state a theorem underlining the necessity of restricting ourselves to non-point compacts. We then construct a continuous convexity measure and compare it to existing measures.
\end{abstract}

\section{Introduction}

\quad Convexity and continuity are central concepts in mathematical analysis, offering powerful tools for investigating and understanding mathematical structures.
Convexity, rooted in geometry, defines the nature of sets and functions, capturing notions of "upward-curving" behavior.
Continuity, on the other hand, focuses on the smoothness and uninterrupted nature of functions

\quad Convexity is characterized by the property that any line segment connecting two points within the set lies entirely within the set itself. 
This intuitive notion of \enquote{no shortcuts} lends itself to numerous applications in optimization, where convex functions and convex sets facilitate efficient algorithms for finding global optima. The concept of convexity extends beyond classical geometry to vector spaces and functional analysis, 
where convex functions and convex combinations find utility in a wide range of disciplines, computer including graphics, operations research, and machine learning.

\quad Continuity, on the other hand, deals with the smoothness and coherence of functions. A function is said to be continuous if arbitrarily small changes in the input result in small changes in the output. Continuity allows stability analysis, where small perturbations in system inputs lead to small perturbations in outputs, a critical consideration in engineering, physics, and complex systems.



\subsubsection{Analysis of Optimization Problems:}
Convexity plays a crucial role in optimization problems. Continuity in convexity ensures that small changes in the objective function or constraints of a convex optimization problem result in small changes in the optimal solution. This property enables the use of efficient algorithms to find the optimal solution and provides guarantees on the quality of the solution obtained.

\subsubsection{Stability Analysis:} Continuity in convexity allows for stability analysis in mathematical models. Small perturbations or uncertainties in the data or parameters of a convex model lead to small perturbations in the solutions. This property is valuable for assessing the robustness and sensitivity of solutions to changes in the problem inputs.

\subsubsection{Convergence Analysis:} Continuity in convexity facilitates the convergence analysis of optimization algorithms. When the objective function or constraints are continuously convex, various optimization algorithms are known to converge to the optimal solution. This allows for efficient numerical methods to solve convex optimization problems.

\subsubsection{Modeling and Decision-Making:} Convexity is widely used in mathematical modeling to represent real-world problems accurately. Continuity in convexity allows for the modeling of smooth and continuous relationships between variables, enabling better understanding and interpretation of the underlying phenomena. 

\quad This article constructs and give examples of continuous \enquote{convexity measures}. Let us clarify what we mean by this. A convexity measure is a way of measuring the convexity defect of a (compact) part of $\R^n$. 

\quad Numerous methods for measuring the convexity defect exist in the literature: \cite{zunic2004new,stern1989polygonal,rosin2007probabilistic}, the most common hypotheses -- and those which we will use in particular -- in \cite{zunic2004new}. Usually those methods are compared to each other empirically, rather than theoretically.

\quad In this paper, we first define the desirable assumptions, in particular continuity in the Hausdorff sense. We then give a result restricting convexity measures to non-point compacts. Then we present an example of a convexity measure. This measure will be compared to other measures found in the literature, notably via their monotony on $\ell^p$ balls of $\R^2$ for $p\in[0,1]$.\nl

\section{Definitions}
\subsection{Convexity measure and properties studied}
Let $\textbf{C}(\R^n)$ be the set of non-empty compacts of $\mathbb{R}^n$.
We call a \textsl{convexity measure} a function $m: \textbf{C}(\R^n)\mapsto \R$ satisfying the following criteria:
\begin{description}
\item [\textbf{Standardisation 1:}] $m: \textbf{C}(\R^n) \to [0,1] $
\item [\textbf{Standardisation 2:}] $\forall A\in \textbf{C}(\R^n)$, $m(A)=1$ if and only in
if $A$ is convex.
\item [\textbf{Continuity:}] a measure $m: \textbf{C}(\R^n) \to \mathbb{R} $ is continuous if $m(A_n) \xrightarrow[A_n \to A]{} m(A)$ 
\item [\textbf{Stability by similarities:}] the convexity measure is invariant by similarities.
\item [\textbf{Proximity:}] $\exists (A_n)_n \in \textbf{C}(\R^n)$ such as $A_n \xrightarrow[n \to \infty]{} 0$
\end{description}

We will also study the following property, compliance to which is left to be decided according to the use the measurement:\nl
\textbf{\underline{Stability through remoteness}:} Let $C_1, ..., C_k$ the connected components of $A\subset \mathbb{R}^n$. The method $m$  is said to be remotely stable if from a certain distance of the components $C_1, ..., C_k$, the method is constant.\nl

\subsection{Hausdorff's distance}

Hausdorff distance \cite{has} is a fundamental metric in mathematics and computational geometry used to quantify the difference between two sets or shapes. Named after the mathematician Felix Hausdorff, this distance measures the maximum distance between any point in one set and its closest point in the other set. It captures the extent of mismatch or deviation between the two sets, providing a comprehensive measure of their difference. Hausdorff distance finds extensive applications in image processing, pattern recognition, shape analysis, and computer vision, where it enables tasks such as shape matching, object recognition, and evaluation of image segmentation algorithms. By quantifying the difference between sets, Hausdorff distance plays a crucial role in solving geometric problems and enhancing our understanding of the relationships and similarities among objects in various domains.

In this section, we recall the usual distance on $\R^n$ compacts: the Hausdorff distance $d_H$ \cite{has,rockafellar2009variational}, defined as follows:\smallskip

For all $A,B\in\textbf{C}(\R^n)$, let:\smallskip

$$d_H(A,B) = \max \left\{\sup_{a\in A}d(a,A) \;;\; \sup_{b\in B}d(b,B)\right\}$$

This defines a distance on $\textbf{C}(\R^n)$.\nl\nl

The Hausdorff distance is characterized as follows: For $K \in\textbf{C}(\R^n)$ and $\varepsilon > 0$, we denote by $V_\varepsilon(K)$ the $\varepsilon$-neighbourhood of $K$, i.e.:
$$V_{\varepsilon}(K) = \{x\in \R^n\;,\;\;d(x,K) \leq \varepsilon \}$$
It can be seen that if $A,B\in\textbf{C}(\R^n)$, then $d_H(A,B) < \varepsilon$ is equivalent to  $A\subset V_\varepsilon(B)$ and $B\subset V_\varepsilon(A)$.\nl\nl
	Let us now note a fundamental property of distance $d_H$:
	\begin{lemm}
		The set $\mathcal{F}$ of finite parts of $\R^n$ is dense in $\textbf{C}(\R^n)$ for distance $d_H$.
	\end{lemm}
	\begin{proof}
	Let $K\in\textbf{C}(\R^n)$. We consider the family $\left(B(x,\varepsilon)\right)_{x\in K}$ of open balls of radius $\varepsilon$ centred at the points of $K$. 
 
 Since:
	$$K\subset \bigcup_{x\in K}B(x,\varepsilon)$$
	is compact, there exists a finite $F\subset K$ such that:
	$$K\subset\bigcup_{x\in F}B(x,\varepsilon)$$
Then, $\forall y\in K$, $\exists x\in F$ such that $y\in B(x,\varepsilon)$ is $d(x,y) < \varepsilon$, which ensures that $K\subset V_\varepsilon(F)$. In the same way, as $F\subset K\subset V_\varepsilon(K)$, we have $d_H(K,F) < \varepsilon$. Therefore the set $\mathcal{F}$ of finite parts of $R^n$ is dense in $\textbf{C}(\R^n)$ for distance $d_H$. 
	\end{proof}

\section{Existence and construction of a convexity measure on the space of non-point compacts}

\subsection{A nonexistence theorem}
\begin{theo}
There is no convexity measure on $\textbf{C}(\R^n)$ which is both continuous for the Hausdorff distance, and invariant by translations and homotheties.
\end{theo}
\begin{proof}
Assume that there exists such a convexity measure $f: \textbf{C}(\R^n) \rightarrow [0,1]$. Let us show that $f$ is constant and equal to $1$. For $n\in\N^*$ or $K\in\textbf{C}(\R^n)$, by homothetic invariance of the measure $f$, we have:
$$f(K) = f\left(\frac{1}{n}K\right)\;\;\;(*)$$ 
As $K$ is compact, $K$ is bounded by a constant $M > 0$. And thus $\forall\varepsilon > 0$, we have $\frac{1}{n}K\subset B(0,\varepsilon) = V_\varepsilon(\{0\})$ as soon as $n > M$. Moreover, $\exists x\in K$ (the elements of $\textbf{C}(\R^n)$ are assumed to be non empty). From then on, $\forall n > \frac{||x||}{\varepsilon}$, we have $d(0,\frac{x}{n}) < \varepsilon $ from which $0 \in V_\varepsilon(K)$.\nl
Thus, for $n\in \N^*$ large enough, we have $\{0\}\subset V_\varepsilon\left(\frac{1}{n}K\right)$ and $\frac{1}{n}K\subset V_\varepsilon\left(\{0\}\right)$, so $d_H\left(\{0\}, \frac{1}{n}K\right) < \varepsilon$. Thus, we have shown that the sequence $\left(\frac{1}{n}K\right)_{n\geq 1}$ converges to $\{0\}$ for the distance $d_H$. By letting $n$ go to $+\infty$ in the equality $(*)$, we obtain by continuity of the function $f$:
$$f(K) = f(\{0\}) = 1$$
because $f(\{0\})$ is convex. This contradicts the property \enquote{arbitrary approach of 0} thereby proving the stated result. 
\end{proof}

The problem in defining a convexity measure on $\textbf{C}(\R^n)$ is the existence of compacts of the plane which are reduced to a point. We are therefore restrict ourselves to the set:
$$\textbf{C}'(\R^n) = \textbf{C}(\R^n)\backslash\{\{x\}\;,\;x\in \R^n\}$$
of $\R^n$ compacts different than point. We will show that on this set, there exists a convexity measure which is both continuous for the Hausdorff distance, and invariant by similarities.
	
\subsection{The set of convexes is Hausdorff-closed}
Let us denote by $\textbf{C}_{\text{Conv}}(\R^n)$ the convex parts of $\R^n$.

\begin{lemm}
The set $\textbf{C}_{\text{Conv}}(\R^n)$ is a closed set of $\textbf{C}(\R^n)$ for the topology induced by the Hausdorff distance.

\end{lemm}
\begin{proof}
Let us show that the complement of $\textbf{C}_{\text{Conv}}(\R^n)$ is open in $\textbf{C}(\R^n)$. Let $K$ be an element of the complement of $\textbf{C}_{\text{Conv}}(\R^n)$. $K$ is therefore not convex, and $\exists x_0,y_0 \in K$ and $\lambda \in[0,1]$ such that $z_0:= \lambda x_0 + (1-\lambda)y_0\not\in K$. By compactness of $K$, there even exists $\varepsilon > 0$ such that $z_0\notin \overline{V_{\varepsilon}(K)}$. The application:
	\begin{align*}
		K\times K&\rightarrow \R^n\\
		(x,y)&\mapsto \lambda x +(1-\lambda)y
	\end{align*}
	being continuous and $\R^n\backslash \overline{V_{\varepsilon}(K)}$ being open, $\exists \eta > 0$ such that $\forall (x, y)\in (\R^n)^2$ verifying $d(x_0,x) < \eta$ and $d(y_0,y)< \eta$, we have $\lambda x+(1-\lambda )y\notin \overline{V_{\varepsilon}(K)}$. \nl
Let $A$ then be $ \in \textbf{C}(\R^n)$ such that $d_H(A,K) < \min(\eta,\varepsilon)$. Let us show that $A$ is not convex. Since $d_H(A,K) < \eta$, then $K\subset V_{\eta}(A)$ and thus $\exists x_1\in A$ such that $d(x_0,x_1) < \eta$. Similarly, $\exists y_1\in A$ such that $d(y_1,y) < \eta$. Then, by definition of $\eta$, we have $z_1 = \lambda x_1 + (1-\lambda)y_1\notin \overline{V_{\varepsilon}(K)}$. But then, in particular, as $d_H(A,K) < \varepsilon$ we get that $A\subset \overline{V_{\varepsilon}(K)}$, this proves that $z_1\notin A$. Thus, we have $x_1,y_1\in A$ but $\lambda x_1 + (1-\lambda)y_1\in A$ i.e. $A$ is not convex.
As for any $K \in \textbf{C}(\R^n)\backslash \textbf{C}_{\text{Conv}}(\R^n)$ the open ball for distance $d_H$ centered in $K$ and of radius $\min(\eta, \varepsilon)$ is contained in the complement of $\textbf{C}_{\text{Conv}}(\R^n)$, then $\textbf{C}_{\text{Conv}}(\R^n)$ is a closed set of $\textbf{C}(\R^n)$ which concludes the proof.
\end{proof}
	
\subsection{Diameter continuity}
The second step in the construction of our continuous and similitude-invariant convexity measure is the introduction of the \textsl{diameter function}. We define, for all $K \in \textbf{C}(\R^n)$, the diameter of $K$, denoted $\text{diam}(K)$, as follows:$$\text{diam}(K) = \sup_{x,y\in K}d(x,y) = \max_{x,y\in K}d(x,y)$$
the sup being well reached by compactness of $K\times K$ and continuity of the application $(x,y)\mapsto d(x,y)$.\nl
Let us now show the continuity of the application diameter:
\begin{lemm}
The application $\text{diam}(\textbf{C}(\R^n))\rightarrow \R_+$ is continuous for the Hausdorff distance.
\end{lemm}
\begin{proof}
By compactness of $A$, for $A,B \in \textbf{C}(R^n)$ and $A,B \in \textbf{C}(R^n)$ such that $d_H(A,B)\leq \varepsilon $, $\exists x,y\in A$ such that $d(x,y) = \text{diam}(A)$.
From then on, as $A\subset  V_{\varepsilon }(B)$, $\exists x',y'\in B$ such that $d(x,x') < \varepsilon$ and $d(y,y') < \varepsilon$. We then have:
$$\text{diam}(B) \geq d(x',y')\geq d(x,y) - d(x,x')-d(y,y') =$$
$$ \text{diam}(A) - 2\varepsilon$$
And therefore: $\text{diam}(A) - \text{diam}(B) < 2\varepsilon$. By switching the roles of $A$ and $B$, we obtain:
$$| \text{diam}(A)-\text{diam}(B)| < 2\varepsilon$$
This being valid for all $A,B$ and $\textbf{C}(\R^n)$ such that $d_H(A,B) < \varepsilon$, we have shown that $\text{diam}$ is a $2$-Lipschitz application, thus continuous for distance $d_H$ which concludes the proof.
\end{proof}

\subsection{Existence of the measure}
\begin{theo}
There is a convexity measure $f: \textbf{C}(\R^n)\rightarrow [0,1]$ which is continuous for the Hausdorff distance.
\end{theo}
\begin{proof}
To optimally position our purpose let us restate some about metric spaces. Recall that a metric space consists of a set of points and a distance function, called a metric, that assigns a non-negative real number to each pair of points. This metric satisfies certain properties, including non-negativity, symmetry, and the triangle inequality, which captures the idea that the distance between two points is always shorter than or equal to the sum of the distances between those points and a third point. 

Given a metric space $(X,d)$, we can define for any part $A \subseteq X$ the function $d_A$, \enquote{distance to $A$}, by:
$$\forall x\in A,\;\;d_A(x) = \inf_{y\in A}d(x,y)$$
It is easy to see that the function $d_A$ is Lipschitz, and therefore continuous, and that it cancels exactly at the points of $\overline A$. In particular, if $A$ is closed in $X$, then $d_A$ is a continuous function whose set of cancellation points is exactly $A$.
Now apply this to the closed $\textbf{C}_{\text{Conv}}(\R^n)$ of $\textbf{C}(\R^n)$. The function:
	\begin{align*}
		\textbf{C}(\R^n)&\rightarrow \R_+\\
		K &\mapsto d_{H,\textbf{C}_{\text{Conv}}(\R^n)}(K)
	\end{align*}
		is continuous, positive, and takes the value $0$ exactly on the set $\textbf{C}_{\text{Conv}}(\R^n)$. As $\text{diam}$ is continuous and does not cancel on $\textbf{C}'(\R^n)$, we can define a continuous function of $\textbf{C}'(\R^n)$ in $[0,1]$, taking the value $1$ exactly on the convexes of $\R^n$ not reduced to a point:
	\begin{align*}
		f: \textbf{C}'(R^n) &\rightarrow [0,1]\\
		K&\mapsto \frac{\text{diam}(K)}{\text{diam}(K) + d_{H,\textbf{C}_{\text{Conv}}(\R^n)}(K)}	
\end{align*}
It remains to see that the function $f$ defined so is invariant by $\R^n$-similarities. For isometries, this is an immediate consequence of the fact that $\text{diam}(K)$ and $d_H$ are invariant by isometries of $\R^n$. For homotheties the invariance follows from:
$$\text{diam}(\lambda K) = \lambda\text{diam}(K)$$$$ d_{H,\textbf{C}_{\text{Conv}}(\R^n)}(\lambda K) = \lambda d_{H,\textbf{C}_{\text{Conv}}(\R^n)}( K) $$
$\forall K\subset \R^n$ compact and all $\forall\lambda \geq 0$
\end{proof}

We therefore define $m_\text{cihi}$ a continuous convexity measure invariant by similarities by:
$$\boxed{\forall K\in\textbf{C}'(\R^n)\;,\;\;\; m_\text{cihi}(K) = \frac{\text{diam}(K)}{\text{diam}(K) + d_{H,\textbf{C}_{\text{Conv}}(\R^n)}( K) }}$$
where:
$$d_{H,\textbf{C}_{\text{Conv}}(\R^n)}( K) =$$$$\inf_{C \text{convex compact}}\inf\{\varepsilon > 0\;|\;K\subset V_{\varepsilon}(C)\;\land C\subset V_{\varepsilon}(K)\}  $$

\section{Catalogue \& comparison of some convexity measures in $\mathbb{R}^2$}
We now present several other convexity measures, some of which are already present in the literature \cite{zunic2004new,rosin2006symmetric}, and compare them. These measures have weaker properties than those of the \underline{c}ontinuous convexity measure \underline{i}nvariant by \underline{h}omotheties and \underline{i}sometries -- which we will note $m_\text{cihi}$. These convexity measures are then compared on $\mathbb{R}^2$ compacts which are the $\ell^p$ balls, whose \enquote{intuitive} level of convexity is parameterized by the real $p\in [0,1]$. In a last step we will consider the possibility of computing these measures efficiently.\smallskip

It will appear that the measure $m_\text{cihi}$, although endowed with good properties, is not easily computable with respect to current results, which restricts its practical use.\\

\subsection{Catalogue and notations}
In this section, we consider $A\in \mathbf{C'}(\mathbb{R}^n)$. Moreover $A\subset \mathbf{R}^n$ we will note $\text{co}(A)$ the convex envelope of $A$. Finally we will denote by $\lambda$ the Lebesgue measure in $\mathbb{R}^n$. \smallskip

The following is a non-exhaustive catalogue of convexity measures found in the literature, or of intuitive interest. \\
 
First, the continuous similitude-invariant measure:
$$\boxed{m_\text{cihi}(A) = \frac{\text{diam}(A)}{\text{diam}(A) + d_{H,\textbf{C}_{\text{Conv}}(\R^n)}( K) }}$$
where: 
$$d_{H,\textbf{C}_{\text{Conv}}(\R^n)}( K) =$$
$$\inf_{C \text{ compact convex}}\inf\{\varepsilon > 0\;|\;K\subset V_{\varepsilon}(C)\;\land C\subset V_{\varepsilon}(K)\}$$\\\\

We are also interested in:

\begin{enumerate}
\item $m_\text{prob}(A)=$ $$\mathbb{P}([X,Y] \subset A) \in [0,1] \quad X,Y \sim \mathcal{U}(A) \ \text(iid)$$ 
defined in \cite{rahtu2006new}
\item $m_\text{env}(A)=$ $$\frac{\lambda(A)}{\lambda({\text{co}(A)})} \in [0,1]$$
\item $m_\text{maxdist}(A)=$ $$\frac{-1}{1+ \underset{x\in \text{co}(A)\backslash A }{\sup d(x,A)}}+1 \in [0,1]$$

\item $m_\text{CE}(A)=$ $$\underset{\underset{E \text{ convex}}{, E\subset \text{co}(A)\backslash A}}{\sup} \frac{\lambda(E)}{\lambda(\text{co}(A))} \in [0,1]$$
\item $m_\text{CI}(A)=$ $$\underset{E \text{convex},\  E\subset A}{\inf} \frac{\lambda(E)}{\lambda(A)} \in [0,1]$$ defined for $\lambda(A)\neq 0$.
\end{enumerate}

\newpage
\subsection{Study of convexity measures on $\ell^p(\mathbb{R}^2)$ balls}
We are interested here in the above-mentioned convexity measures of the unit balls for the norms $\vert . \vert \vert _p$ for $p<1$. The main property of interest in this paragraph is the monotonicity as a function of $p$ of the convexity measures of these balls.
\begin{figure}[H]
  	\centering
	\includegraphics[width=100mm]{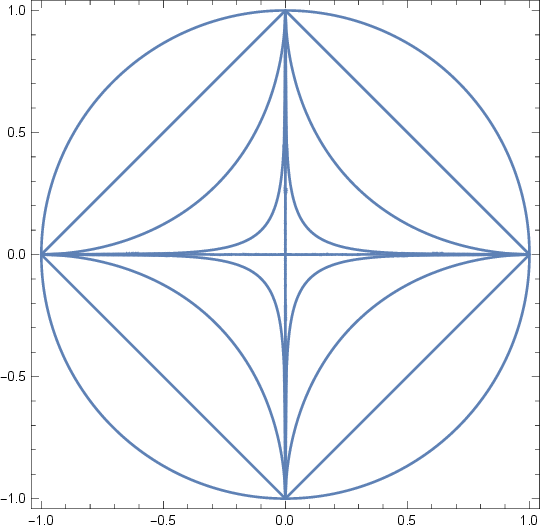}
	\caption{$\ell^p$ balls for $p\in \{ 0.1, 0.3, 0.5, 1, 2\}$}
	\label{fig1}
\end{figure}
The $\ell^p$ balls for $p \geq 1$ are convex. In the case where $p is]0,1[$, they are not, but even more so, the level of pinching and hence of non-convexity seems to increase as $p \to 0$ as shown in \ref{fig1}.

The measures of the balls are expected to converge to a certain limit as $p \to 0$: see \ref{fig2} and this monotonically, with a speed dependant of the measure.

However, this is not the case for all the proposed measures, including the measure $m_\text{prob}$ which, apart from tending to $1/2$ for $p \to 0$, is non-monotonic as shown in \ref{fig3}.

\begin{figure}[H]
  	\centering
	\includegraphics[width=120mm]{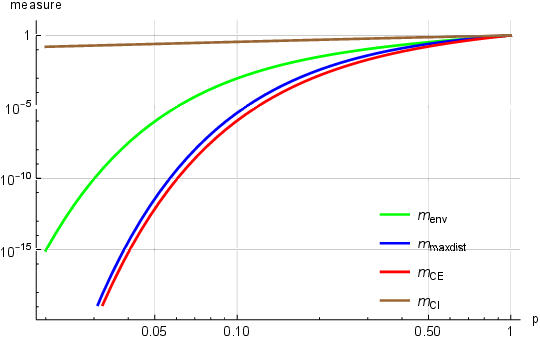}
	\caption{Measures of convexity according to $p$ (log scale) of $\ell^p$ balls}
	\label{fig2}
\end{figure}

\begin{figure}[H]
  	\centering
	\includegraphics[width=120mm]{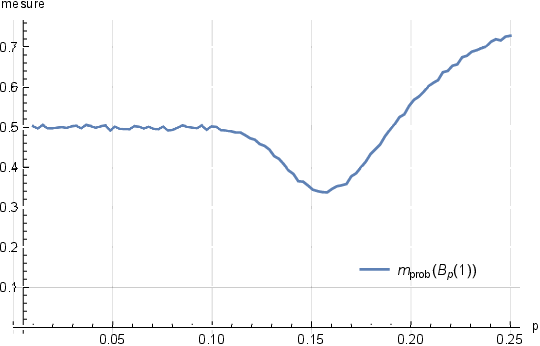}
	\caption{Measurement of $\ell^p$ balls by $m_\text{prob}$}
	\label{fig3}
\end{figure}

\begin{figure}[H]
  	\centering
	\includegraphics[width=120mm]{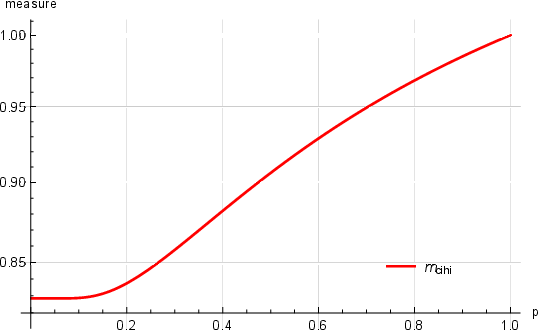}
	\caption{Measurement of $\ell^p$ balls by $m_\text{cihi}$}
	\label{fig4}
\end{figure}

\begin{figure}[H]
  	\centering
	\includegraphics[width=120mm]{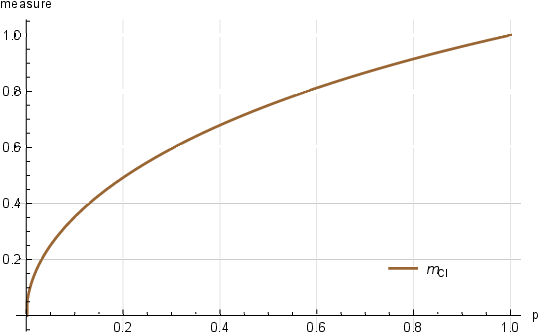}
	\caption{Measurement of $\ell^p$ balls by $m_\text{CI}$}
	\label{fig5}
\end{figure}

By denoting $L_p$ the unit ball for the norm $p$, we obtain the following expressions for convexity measures as a function of $p$ which are plotted in figures \ref{fig2} and \ref{fig4}:
$$m_\text{cihi}(L_p)=\frac{\sqrt{2}+1}{\sqrt{2}+3 / 2-\sqrt[p]{1 / 2}}$$
$$m_\text{CE}(L_p)=\frac{\Gamma\left(\frac{1}{p}\right)^{2}}{\Gamma\left(\frac{2}{p}\right)}$$
$$m_\text{maxdist}(L_p)=4^{1-\frac{1}{p}}$$
$$m_\text{env}=\frac{1}{2^{1/p}-1}$$
$$m_\text{CI}=\frac{2^{2-\frac{2}{p}} p\Gamma\left(\frac{2}{p}\right)}{\Gamma\left(\frac{1}{p}\right)^{2}}=\frac{2^{2-\frac{1}{p}} p}{m_\text{CE}(L_p)}$$

\subsection{Comparison of properties of catalogued measures}
The following functions are continuous:
\begin{itemize}
\item $B: \left\{
  \begin{array}{clc}
    \mathbb{R}^n \times \mathbb{R}_+ & \to & \mathbf{C}(\mathbb{R}^n) \\
   (x,r) & \mapsto & \{y\in  \mathbb{R}^n \vert d(x,y)<r\} \\
  \end{array}
\right.$
\item $\text{co}: \left\{
  \begin{array}{clc}
    \mathbf{C}(\mathbb{R}^n) & \to & \mathbf{C}(\mathbb{R}^n) \\
   A & \mapsto & \underset{A \subset E, E \text{convex} }\bigcap E \\
  \end{array}
\right.$

\item $\lambda$ the Lebesgue measure on $\mathbf{C}(\mathbb{R}^n)$.
\item $ d$ the Euclidean distance on $\mathbb{R}^n$.
\item For any real continuous $\varphi$ function, is continuous: 
$\left\{
  \begin{array}{clc}
    \mathbf{C'}(\mathbb{R}^n) & \to & \mathbb{R} \\
   A & \mapsto & \underset{E \text{convex} \subset A }\sup \varphi(E) \\
  \end{array}
\right.$
\end{itemize}

In general, the convexity measures defined from a part evolving in $\text{co}(A)\backslash A$ are not stable by distance, the size of the convex envelope being able to increase while leaving the Lebesgue measure of $A$ unchanged. \emph{A contrario}, the measures defined by objects evolving in $A$ are stable by distance.

Detailed results are summarised in Table \ref{fig4}.

\begin{enumerate}
\item $m_\text{cihi}$ is stable by similarities and continuous. The calculation of this measure is however more difficult. 

\item $m_\text{prob}$ is stable by moving away the related components from a certain distance between the parts. This measure has the advantage of being easily computable by Monte-Carlo methods when the ratio between the measure of the smallest paving stone containing $A$ and the measure of $A$ is small (i.e. $\simeq 1$), in the opposite case (example of the ball $l^{0.1}$ \ref{fig1}), the difficulty of carrying out a sampling within the set makes the computation cumbersome without additional refinements. $m_\text{prob}$ may also be computable by an analogous principle via a grid approximation. $m_\text{prob}$ is stable by isometries and homotheties, but not continuous. The non-continuity follows from the argument of two half-balls tending to the same ball.

\item $m_\text{env}$ is not stable by distance, the Lebesgue measure of the convex envelope increasing as the distance between the connexal components increases. $m_\text{env}$ is however stable by isometries and homotheties, by the same argument as for $m_\text{prob}$. $m_\text{env}$ is continuous because $\text{co}$ and $\lambda$ are continuous on $\mathbf{C'}(\mathbb{R}^n)$. The calculation of the convex envelope can be done with the QuickHull algorithm \cite{qh} in mean time $O(n\log(n))$, followed by two volume measurements via a Monte-Carlo method as shown for $m_\text{prob}$.

\item $m_\text{maxdist}$ is continuous because $\text{co}$ and $d$ are continuous, unstable by remoteness and not stable by homotheties.

\item $m_\text{CE}$ is continuous by continuity of $\lambda$ and $\text{co}$, unstable by remoteness, and stable by isometries.

\item $m_\text{CI}$ is stable by similarities, continuous,  unstable by remoteness. However, $m_\text{CI}$ is not well defined for sets of measure zero.

\end{enumerate}

\begin{figure}[h]
	\centering
	\begin{tabular}{|c|c|c|c|c|c|c|}
\hline~~\textbf{property}~~&~~$m_\text{cihi}$~~&~~$m_\text{prob}$~~&~~$m_\text{env}$~~&~~$m_\text{maxdist}$~~&~~$m_\text{CE}$~~& ~~$m_\text{CI}~~$ 
\\\hline~~continuity~~& \cmark & \xmark & \cmark & \cmark & \cmark & \cmark 
\\\hline~~stability by isometries~~& \cmark & \cmark & \cmark & \cmark & \cmark & \cmark 
\\\hline~~stability by homotheties~~& \cmark & \cmark & \cmark & \xmark & \cmark & \cmark 
\\\hline~~stability by remoteness~~& \cmark & \cmark & \xmark & \xmark & \xmark & \cmark 
\\\hline~~large definition domain~~& \cmark & \cmark & \cmark & \cmark & \cmark & \xmark 
\\\hline~~monotonicity on the $\ell^p$-balls~~& \cmark & \xmark & \cmark & \cmark & \cmark & \cmark 
\\\hline \end{tabular}
	\caption{Summary of properties of catalogued measures}
    \cmark : the property is fulfilled ; \xmark : the property is not fulfilled
	\label{fig6}
\end{figure}

\section{Conclusion}
We have introduced a novel measure, denoted as $m_\text{cihi}$, which possesses the desirable properties of a convexity measure, specifically continuity. By carefully constraining the domain of definition, we have achieved the continuity requirement. Comparing $m_\text{cihi}$ with existing convexity measures from the literature, we have discovered that it outperforms them in terms of meeting the defined criteria for convexity measurement.

Expanding our perspective, we can delve into the characterization of all measures that satisfy the assumptions of a continuous convexity measure from a different standpoint. Are there families of measures that exhibit these properties, or can we parameterize them to gain a deeper understanding? Additionally, an intriguing question arises concerning the computational efficiency of the measures within this family. Can we identify which measures can be computed efficiently, thereby paving the way for practical applications and computational implementations? By exploring these avenues, we can further enhance our understanding of convexity measures and uncover valuable insights for both theoretical analysis and practical usage.

\bibliographystyle{plain}
\bibliography{bibliographie}

\end{document}